\documentstyle[amssymb,amsfonts,12pt]{amsart}

\newtheorem{theorem}{Theorem}[section]
\newtheorem{lemma}[theorem]{Lemma}

\newtheorem{proposition}[theorem]{Proposition}
\newtheorem{definition}[theorem]{Definition}

\theoremstyle{remark}

\def\QSet{\mbox{\rm\kern.24em
\vrule width.03em height1.48ex depth-.051ex \kern-.26em Q}}

\def\Z{{\mathbb Z}}
\def\R{{\mathbb R}}

\def\G{{\mathcal G}}

\def\bas{\begin{align*}}
\def\eas{\end{align*}}
\def\bi{\begin{itemize}}
\def\ei{\end{itemize}}
\newenvironment{proof}{\noindent {\bf Proof} }{\endprf\par}
\def \endprf{\hfill  {\vrule height6pt width6pt depth0pt}\medskip}
\def\emph#1{{\it #1}}

\begin{document}
\title{Singular integrals along $N$ directions in $\R^2$}

\author{Ciprian Demeter}
\address{Department of Mathematics, Indiana University, 831 East 3rd St., Bloomington IN 47405}
\email{demeterc@@indiana.edu}

\keywords{Kakeya maximal function, singular integrals}
\thanks{The author is supported by a Sloan Research Fellowship and by NSF Grants DMS-0742740 and 0901208  }
\begin{abstract}
We prove optimal bounds in $L^2(\R^2)$ for the maximal operator obtained by taking a singular integral along $N$ arbitrary directions in the plane. We also give a new proof for the optimal $L^2$ bound for the single scale Kakeya maximal function. 
\end{abstract}

\maketitle

%%%%%%%%%%%%%%%%%%%%%%%%%%%%%%
\section{Introduction}
Our main result is the following:

\begin{theorem}
\label{T1}
Let $\Sigma_N$ a set of $N$ unit vectors in $\R^2$. Let $m$ be a  H\"ormander-Mikhlin multiplier, that is
$$|\partial^{\alpha}m(\xi)|\lesssim_{\alpha}\frac{1}{|\xi|^{\alpha}},\;\;\xi\not=0$$
for each $\alpha\ge 0$. Define
$$T_{N}^{*}f(x,y):=\sup_{v\in\Sigma_N}|\int \hat{f}(\xi,\eta)m(v_1\xi+v_2\eta)e^{i(x\xi+y\eta)}d\xi d\eta|.$$
Then $$\|T_{N}^{*}\|_{2\to 2}\lesssim \log N.$$
\end{theorem}

In Section \ref{S3} we also prove that the bound $\log N$ is optimal. It is well known (see for example Proposition 2, page 245 in \cite{ST}) that the distribution $\hat{m}$ coincides with a function $K$ outside the origin, which means that if $f$ is compactly supported and  $(x,y)$ is outside the support of $f$, then we also have the following kernel representation

$$T_{N}^{*}f(x,y)=\sup_{v\in\Sigma_N}|\int f((x,y)+tv)K(t)dt|$$

Thus, Theorem \ref{T1} is the singular integral analogue of the following bound for the Kakeya maximal function due to Nets Katz:

\begin{theorem}[\cite{K}]
\label{T5}
Let $\Sigma_N$ a set of $N$ unit vectors in $\R^2$. Define
$$M_{N}^{*}f(x,y):=
\sup_{v\in\Sigma_N}\sup_{\epsilon>0}|\frac{1}{2\epsilon}\int_{|t|<\epsilon} f((x,y)+tv)dt|.$$
Then $$\|M_{N}^{*}\|_{2\to 2}\lesssim \log N.$$
\end{theorem}

In section \ref{S4} we give a new proof for the following fact proved in \cite{K1}
\begin{theorem}
\label{T7}
Let $\Sigma_N$ a set of $N$ unit vectors in $\R^2$. Define the single scale Kakeya operator
$$M_{N}^{0}f(x,y):=
\sup_{v\in\Sigma_N}|\int_{|t|<1} f((x,y)+tv)dt|.$$
Then $$\|M_{N}^{0}\|_{2\to 2}\lesssim \sqrt{\log N}.$$
\end{theorem}

The proofs in both \cite{K} and \cite{K1} rely on stopping time arguments similar to John Nirenberg's inequality in the context of product BMO.
Perhaps surprisingly, our arguments in this paper avoid product theory, and are essentially $L^2$ based.

 The proof of the Theorem \ref{T1} is very similar to the proof of Theorem 1.1. in \cite{GHS}, in particular it relies on the deep inequality of Chang, Wilson and Wolff. The second major ingredient we use is a result of Lacey and Li \cite{LL} for single annulus operators, which is essentially equivalent with Carleson's theorem
  \cite{Car}.

In the case $K(t)=p.v.\frac1t$, the bound
$$\|H_{N}^{*}\|_{2\to 2}\lesssim \log N.$$
for the operator
$$H_{N}^{*}f(x,y):=$$
$$\sup_{v\in\Sigma_N}|p.v.\int f((x,y)+tv)\frac{dt}{t}|=\sup_{v\in\Sigma_N}|\int \hat{f}(\xi,\eta)sgn(v_1\xi+v_2\eta)e^{i(x\xi+y\eta)}d\xi d\eta|$$
immediately follows from the Rademacher-Menshov theorem. We do not see a way to extend this type of approach to the general case.

We would like to thank Zubin Gautam, Nets Katz and Francesco Di Plinio for stimulating discussions on the subject.

\section{Discrete versus continuous square function}

To simplify notation, we will often replace $(x,y)$ with $x$.

Recall the conditional expectation with respect to the $\sigma$- algebra  consisting of dyadic squares of side length $2^{-j}$,
$$E_jf(x):=\sum_{Q:|Q|=2^{-2j}}\langle f,\frac{1_Q}{|Q|}\rangle 1_Q(x)$$
 and let $\Delta_j=E_{j+1}-E_{j}$ be the martingale difference. Denote by
 $$\Delta(f)(x)=(\sum_{k\ge 0}|\Delta_kf(x)|^2)^{1/2}$$
 the discrete square function. We will need the following fundamental inequality.

\begin{lemma}[The Chang-Wilson-Wolff inequality, \cite{CWW}]
\label{L5}
There exist constants $c_1,c_2$ such that for all $\lambda>0$ and $0<\epsilon<1$ one has
$$|\{x\in\R^2:|f(x)-E_0f(x)|>2\lambda,\Delta(f)(x)<\epsilon\lambda\}|\le c_2e^{-\frac{c_1}{\epsilon^2}}|\{x\in\R^2: \sup_{k\ge 0}|E_kf(x)|>\epsilon\lambda\}|.$$
\end{lemma}

Define for each $v=(v_1,v_2)\in\Sigma_N$
$$T_vf(x,y):=\int \hat{f}(\xi,\eta)m(v_1\xi+v_2\eta)e^{i(x\xi+y\eta)}d\xi d\eta.$$

Let $\phi\in C^{\infty}(\R^2)$ be supported in $1/2<|\xi|<2$ so that
$$\sum_{k\in\Z}\phi(2^{-k}\xi)=0,\;\;\xi\not=0.$$
Choose $\beta$ a Schwartz function such that $\hat{\beta}$ is supported in $|x|\le 1/4$ and $\beta$ is nonzero in $1/4\le |\xi|\le 4$ and $\beta(0)=0$. Finally, choose a Schwartz function $\psi$ such that $\psi(\xi)\beta^2(\xi)=1$ for $\xi$ in the support of $\phi.$

For each linear bounded operator $T:L^2\to L^2$ denote by $T_kf=T(S_kf):L^2\to L^2$, where $\widehat{S_kf}(\xi)=\hat{f}\phi(2^{-k}\xi)$.

We need the following variant of Lemma 3.1 from \cite{GHS}.
\begin{lemma}
Let $T$ be a linear bounded multiplier operator $T:L^2\to L^2$, that is
$$\widehat{Tf}=m\hat{f},$$
for some $m\in L^{\infty}(\R^2)$.
 Then
there exists $c_3>0$ independent of $N$, $T$ and $x$ such that for almost every $x\in\R^{2}$ and each $f\in L^2(\R^2)$
$$\Delta(Tf)(x)\le c_3 (\sum_{k\in\Z}|M_2(M(T_kf)(x))|^2)^{1/2},$$ where
$$M_2g(x)=(M(g^2)(x))^{1/2},$$
$Mg$ being the standard Hardy-Littlewood maximal function.
\end{lemma}
\begin{proof}
Define
$$\widehat{B_kf}(\xi)=\beta(2^{-k}\xi)\hat{f}(\xi)$$
$$\widehat{L_kf}(\xi)=\psi(2^{-k}\xi)\hat{f}(\xi).$$
Note that
$$\Delta_kTf=\sum_{n\in\Z}(\Delta_kB_{k+n})(B_{k+n}L_{k+n})T_{k+n}f.$$
Sublemma 4.2 from \cite{GHS} implies that (uniformly) for each $x$
$$|\Delta_kB_{k+n}g(x)|\lesssim 2^{-|n|/2}M_2g(x).$$
Since also
$$|B_{k+n}L_{k+n}g(x)|\lesssim Mg(x),$$
uniformly over $k,n$, the result follows from the triangle inequality.
\end{proof}

\begin{lemma}
\label{L2}
Let $T$ be a linear bounded multiplier operator $T:L^2(\R^2)\to L^2(\R^2)$ associated with multiplier $m$. Assume in addition that $m$ is zero on the ball with radius $2^{2N}$.
Then
there exists $c_4>0$ independent of $N$, $T$ and $x$  such that for almost every $x\in\R^{2}$ and each $f\in L^2(\R^2)$
$$|E_0(Tf)(x)|\le c_42^{-N} (\sum_{k\in\Z}|M_2(M(T_kf)(x))|^2)^{1/2}$$
\end{lemma}
\begin{proof}
The argument follows as in the previous lemma, by using the fact that (part of Sublemma 4.2 in \cite{GHS})
$$|E_0B_{n}g(x)|\lesssim 2^{-|n|/2}M_2g(x),$$
for each $n\ge 0$.
\end{proof}

\section{Proof of the Theorem \ref{T1}}
\label{S3}

We have proved in the previous section that
$$\sup_{v\in\Sigma_N}|\Delta (T_v)f(x)|\le G(f)(x):=c_3(\sum_{k\in\Z}|M_2(M(\sup_{v\in\Sigma_N}|T_v(S_kf)|)(x))|^2)^{1/2}.$$

The following Lemma is a variant of similar lemmas from \cite{CHS} and \cite{K}.
\begin{lemma}
\label{L3}
Let $T$ be a sublinear operator on some measure space $(X,m)$. Assume
$$\|T\|_{L^1\to L^{1,\infty}}\lesssim N$$
$$\|T\|_{L^2\to L^{2,\infty}}\lesssim 1$$
$$\|T\|_{L^p\to L^{p}}\lesssim N,$$
for some $2<p<\infty$.
Then
$$\|T\|_{L^2\to L^{2}}\lesssim \sqrt{\log N}$$
\end{lemma}
\begin{proof}
We need to prove that
$$\int_{0}^{\infty}\lambda|\{Tf>3\lambda\}|d\lambda\lesssim \log N\|f\|_2^2.$$
Let
$$\nu_f(\alpha):=|\{|f|>\alpha\}|.$$
Split
$$f=f_1^{\lambda}+f_2^{\lambda}+f_3^{\lambda},$$
where
\[f_1^{\lambda}=\begin{cases}f,&:\quad \hbox{if } |f|>N\lambda\\ \hfill  0&:\quad \text{otherwise}\end{cases}, \]
\[f_2^{\lambda}=\begin{cases}f,&:\quad \hbox{if } N\lambda\ge |f|>N^{-q}\lambda\\ \hfill  0&:\quad \text{otherwise}\end{cases}, \]
\[f_3^{\lambda}=\begin{cases}f,&:\quad \hbox{if } N^{-q}\lambda\ge |f|\\ \hfill  0&:\quad \text{otherwise}\end{cases}. \]
where $q(p-2)=p.$
Now,
$$\int_{0}^{\infty}\lambda|\{Tf_1^{\lambda}>\lambda\}|d\lambda\lesssim N\int_{0}^{\infty}\|f_1^{\lambda}\|_1d\lambda=N\int_{0}^{\infty}\int_{N\lambda}^{\infty}\nu_f(\alpha)d\alpha d\lambda=$$
$$=N\int_{0}^{\infty}\nu_f(\alpha)\int_{0}^{\alpha/N}d\lambda d\alpha=\|f\|_2^2,$$
$$\int_{0}^{\infty}\lambda|\{Tf_2^{\lambda}>\lambda\}|d\lambda\lesssim \int_{0}^{\infty}\lambda^{-1}\|f_2^{\lambda}\|_2^2d\lambda=\int_{0}^{\infty}\lambda^{-1}\int_{N^{-q}\lambda}^{N\lambda}\alpha\nu_f(\alpha)d\alpha d\lambda=$$
$$=\int_{0}^{\infty}\alpha\nu_f(\alpha)\int_{\alpha/N}^{N^q\alpha}\lambda^{-1}d\lambda d\alpha\lesssim (\log N)\|f\|_2^2$$
$$\int_{0}^{\infty}\lambda|\{Tf_3^{\lambda}>\lambda\}|d\lambda\lesssim N^p\int_{0}^{\infty}\lambda^{1-p}\|f_3^{\lambda}\|_p^pd\lambda=N^{p}\int_{0}^{\infty}\lambda^{1-p}\int_{0}^{N^{-q}\lambda}\alpha^{p-1}\nu_f(\alpha)d\alpha d\lambda=$$
$$=N^p\int_{0}^{\infty}\alpha^{p-1}\nu_f(\alpha)\int_{N^q\alpha}^{\infty}\lambda^{1-p}d\lambda d\alpha\lesssim \|f\|_2^2$$
\end{proof}
\begin{proposition}
\label{P6}
$$\|G(f)\|_2\le C \sqrt{\log N}\|f\|_2,$$
with $C$ independent of $N$.
\end{proposition}
\begin{proof}
It suffices to prove that for each $k$
$$\|A_k(f):=\sup_{v\in\Sigma_N}|T_v(S_kf)|\|_2\le C \sqrt{\log N}\|S_kf\|_2.$$
It was proved in \cite{LL} that $A_k$ maps $L^{2}$ to $L^{2,\infty}$ and $L^p$ to itself for $2<p<\infty$, with implicit bounds independent of $N$. That is a deep result, essentially equivalent with Carleson's theorem \cite{Car}.

Since each $T_v$ maps $L^1$ to $L^{1,\infty}$,
$$\|A_k(f)\|_{1,\infty}\lesssim N\|f\|_1.$$ The result now follows from Lemma \ref{L3} above.
\end{proof}
\begin{proof}[of Theorem \ref{T1}]
By using a limiting argument, we can assume that $m$ is supported in a finite union of dyadic annuli $2^j\le |\xi|<2^{j+1}$. By rescaling, we can assume $m$ is actually supported away from the ball of radius $2^{2N}$ (for example).
Let $\epsilon_N=\sqrt{c_1\log N}$. For each $\lambda>0$,
$$\{x:\sup_{v\in\Sigma_N}|T_vf(x)|>4\lambda\}\subset E_{\lambda,1}\cup E_{\lambda,2}\cup E_{\lambda,3},$$
where
$$E_{\lambda,1}=\{x:\sup_{v\in\Sigma_N}|T_vf(x)-E_0T_vf(x)|>2\lambda,Gf(x)\le \epsilon_N\lambda\}$$
$$E_{\lambda,2}=\{x:Gf(x)> \epsilon_N\lambda\}$$
$$E_{\lambda,3}=\{x:\sup_{v\in\Sigma_N}|E_0T_vf(x)|>2\lambda\}.$$
By Lemma \ref{L5},
$$|E_{\lambda,1}|\le \sum_{v\in\Sigma_N}|\{x:|T_vf(x)-E_0T_vf(x)|>2\lambda,\Delta(T_vf)(x)\le \epsilon_N\lambda\}|\lesssim$$
$$\lesssim \frac{1}{N}\sum_{v\in\Sigma_N}|\{x:M(T_vf)(x)>\epsilon_N\lambda\}|.$$
Using this, Lemma \ref{L2} and Proposition \ref{P6}, we get that
$$\int_{0}^{\infty}\lambda\sum_{i=1}^3|E_{\lambda,i}| d\lambda\lesssim (\log N)^2\|f\|_2^2.$$
Theorem \ref{T1} now follows.
\end{proof}

The following result shows that the bound in Theorem \ref{T1} is optimal.

\begin{proposition}
Let $\Sigma_N$ a set of $N$ equidistributed unit vectors in $\R^2$. Define
$$T_{N}^{*}f(x,y):=\sup_{v\in\Sigma_N}|\lim_{\epsilon\to\infty}\int_{|t|>\epsilon}f((x,y)+tv)\frac{dt}{t}|.$$
Then,
$$\|T_N^{*}\|_{2\to 2}\gtrsim (\log N).$$
\end{proposition}
\begin{proof}
This is a standard example, we only briefly sketch the details. Define $f(x,y)=\frac{1}{\|(x,y)\|}1_{1\le \|(x,y)\|\le N/100}$.

Let $100\le \|(x,y)\|\le N/100$. If $v:=\frac{(x,y)}{\|(x,y)\|}$, then
$$|\lim_{\epsilon\to\infty}\int_{|t|>\epsilon}f((x,y)+tv)\frac{dt}{t}|=\frac{1}{\|(x,y)\|}|\log\frac{N^2-\|(x,y)\|^2}{N^2}-\ln (\|(x,y)\|^2-1)|
$$$$\ge \frac{\log \|(x,y)\|}{2\|(x,y)\|}.$$
Let $w\in \Sigma_N$ be such that $\|v-w\|\le \frac{10}{N}$. It easily follows that
$$|\lim_{\epsilon\to\infty}\int_{|t|>\epsilon}f((x,y)+tv)\frac{dt}{t}-\lim_{\epsilon\to\infty}\int_{|t|>\epsilon}f((x,y)+tw)\frac{dt}{t}|<\frac{\log \|x\|}{10\|x\|},$$
simply from the fact that $\|tv-tw\|\le 1/5$ if $|t|<N/50$ and the fact that $|f(x,y)-f(x',y')|\le \frac{2\|(x,y)-(x',y')\|}{\|(x,y)\|^2}$ if $1\le \|(x,y)\|,\|(x',y')\|\le N/100$ and $\|(x,y)-(x',y')\|\le 1$. Thus,
$$T_N^{*}f(x,y)\ge |\lim_{\epsilon\to\infty}\int_{|t|>\epsilon}f((x,y)+tw)\frac{dt}{t}|\ge \frac{\log \|(x,y)\|}{4\|(x,y)\|},$$
and hence
$$\|T_N^{*}f\|_2\gtrsim (\log N)^{3/2}.$$
Since $\|f\|_2\lesssim (\log N)^{1/2},$ the Proposition follows.
\end{proof}

\section{Proof of Theorem \ref{T7}}
\label{S4}
First, let us remark that there is a proof along the lines of the argument for Theorem \ref{T1}. We however choose to give a different, self contained argument, one that in particular does not appeal to the Chang-Wilson-Wolff inequality.

It suffices to prove the bound for
$$M_{0}F(x,y):=\sup_{v\in\Sigma_N}|\int F(x+tv_1,y+tv_2)\psi(t)dt|,$$
where $\psi$ is a positive function such that $\hat{\psi}$ is supported in $[0,1]$.

Decompose
$$F=F_0+\sum_{n\ge 1}F_n,$$
such that
$$\|F\|_2^2=\|F_0\|_2^2+\sum_{n\ge 1}\|F_n\|_2^2.$$
Here $F_0$ is the Fourier restriction to the unit ball $B$ (that is, $\widehat{F_0}=\hat{F}1_B$), while $F_n$ the Fourier restriction to the annulus $2^{n-1}\le |(\xi,\eta)|\le 2^{n}$. Note that
$$M_0F_0(x,y)=\sup_{v\in \Sigma_N}|\int \hat{F_0}(\xi,\eta)\phi(\xi,\eta)\hat{\psi}(v_1\xi+v_2\eta)e^{i(x\xi+y\eta)}d\xi d\eta|\lesssim MF_0(x,y),$$
where $M$ is the Hardy-Littlewood maximal function, and
$$1_B\le \phi\le 1_{2B}$$
is smooth. This is
since $K_v(x,y):=(\phi(\xi,\eta)\hat{\psi}(v_1\xi+v_2\eta))\check (x,y)$ easily satisfies
$$|K_v(x,y)|\lesssim (1+\|(x,y)\|)^{-3},$$
with bound independent of $v$, and $M_0F_0(x,y)=\sup_{v\in \Sigma_N}|F_0*K_v(x,y)|$.

Next, we analyze the case $n\ge 1$.

\begin{definition}[Grids]
A collection of intervals on the unit circle $S^1$ is called a grid if for any two intervals from the collection that intersect, one of them should contain the other one.
\end{definition}
The standard dyadic grid $\G^0$ is the collection dyadic intervals on $S^1$, obtained by identifying $S^1$ with $[0,1)$. In addition, for $i\in\{1/3,2/3\}$, define
$\G^i$ to consist of the projection on $S^1$ of the intervals
$$\{[2^jk,2^j(k+(-1)^ji)]:2^{j}\le 1,k\in\Z\}.$$

It is easy to see that each $\G^i$ is a  grid. Moreover,  each interval $I$ on the circle is contained inside an interval of similar length from one of these grids.

Denote by $\Omega_n$ any partition of the annulus $2^{n-1}\le |(\xi,\eta)|\le 2^{n}$ in $n$ sectors with equal area. For each ${\omega}\in\Omega_n$, call $d(\omega)$ the direction of the radius that splits $\omega$ in two sectors with equal area. For each $\omega\in\Omega_n$, let $\tilde{\omega}$ be the sector in the larger annulus $2^{n-2}\le |(\xi,\eta)|\le 2^{n+1}$, with twice the aperture
of $\omega$ and having the same bisector $d({\omega})$. Let also $A(\omega)$ be the projection of the sector $\omega$ on $S^1$. Denote by $B(\omega)$  some interval from some grid $\G^i$ that contains $10A(\omega)+\frac{\pi}{2}$\footnote{$10A(\omega)$ is the interval on $S^1$ with the same center as $A(\omega)$ and 10 times its length}, and whose length is comparable to that of $A(\omega)$, which in turn is comparable to $2^{-n}$.

Decompose further
$$F_n=\sum_{\omega\in\Omega_n}F_\omega,$$
where $\widehat{F_\omega}=\hat{F}1_{\omega}.$
Note that
$$\int F_\omega(x+tv_1,y+tv_2)\psi(t)dt=\int\hat{F_\omega}(\xi,\eta)\hat{\psi}(v_1\xi+v_2\eta)e^{i(x\xi+y\eta)}d\xi d\eta$$
is nonzero only if there exists $(\xi,\eta)\in{\omega}$ such that $0\le v_1\xi+v_2\eta\le 1$. This implies that  $v\in 10A(\omega)+\frac{\pi}{2}\subset B(\omega)$.

Let $\theta_\omega$ be a bump function adapted to and supported on $\tilde{\omega}$ and which equals 1 on $\omega$. Note that
$$\int F_\omega(x+tv_1,y+tv_2)\psi(t)dt=\int\hat{F_\omega}(\xi,\eta)\theta_{\omega}(\xi,\eta)\hat{\psi}(v_1\xi+v_2\eta)e^{i(x\xi+y\eta)}d\xi d\eta.$$
\begin{lemma}
\label{L543}
If $v,v'\in B(\omega)$ and $\omega\in \Omega_n$, then
$$|\int F_\omega(x+tv_1,y+tv_2)\psi(t)dt-\int F_\omega(x+tv_1',y+tv_2')\psi(t)dt|\lesssim$$
\begin{equation}
\label{1}
\lesssim 2^n\|v-v'\|M^{*}F(x,y),
\end{equation}
where
$$M^{*}F(x,y)=\sup_{\epsilon,\delta>0}\frac{1}{\epsilon\delta}\int_{|t|<\epsilon}\int_{|s|<\delta}|F|(x+t,y+s)dtds.$$
\end{lemma}
\begin{proof}
Indeed, by rotation invariance it suffices to assume that $d(\omega)=(0,-1)$. Then, $B(\omega)$ is an interval of length roughly $2^{-n}$ containing $(1,0)$. Thus, if $v,v'\in B(\omega)$, then $|v_2|,|v_2'|\lesssim 2^{-n}$. This immediately shows that for each $\alpha,\beta\ge 0$
$$|\partial^{\alpha}_\xi\partial^\beta_\eta[\theta_{\omega}(\xi,\eta)(\hat{\psi}(v_1\xi+v_2\eta)-\hat{\psi}(v_1'\xi+v_2'\eta))]|\lesssim 2^{n(1-\beta)}\|v-v'\|.$$
By using this, integration by parts and the fact that $\theta_{\omega}$ is supported in $\tilde{\omega}$, it follows that
$$|K_{\omega}(x,y):=\int\theta_{\omega}(\xi,\eta)[\hat{\psi}(v_1\xi+v_2\eta)-\hat{\psi}(v_1'\xi+v_2'\eta)]e^{i(x\xi+y\eta)}d\xi d\eta|$$
$$\lesssim 2^{2n}\|v-v'\|\min\{1,|x|^{-2},|y2^n|^{-2}\}\lesssim2^{2n}\|v-v'\|(1+|x|+2^n|y|)^{-2}.$$
The lemma follows.
\end{proof}

Define
$$\tilde{F}_\omega(x,y,v):=\int F_\omega(x+tv_1,y+tv_2)\psi(t)dt,$$
and $\Omega=\cup_{n\ge 1}\Omega_n$.
Thus,
$$M_0(\sum_{n\ge 1}F_n)(x,y)=\sup_{v\in \Sigma_N}|\sum_{n\ge 1}\sum_{\omega\in \Omega}\tilde{F}_\omega(x,y,v)|.$$
By splitting the sum above in three parts, we can assume without loss of generality that all intervals $B(\omega)$ belong to a fixed grid.

For $1\le \kappa\le \log N$, let $\Omega^\kappa$ consist of those $\omega\in\Omega$ such that
\begin{equation}
\label{e.76}
2^{\kappa-1}<|B(\omega)\cap\Sigma_N|\le 2^\kappa.
\end{equation}
 Let
$$T_\kappa F(x,y)=\sup_{v\in \Sigma_N}|\sum_{n\ge 1}\sum_{\omega\in \Omega^{\kappa}}\tilde{F}_\omega(x,y,v)|.$$
It suffices to prove that $\|T_\kappa\|_2\lesssim 1$ for each $\kappa$, by orthogonality and the fact that
$$T_\kappa F=T_\kappa(\sum_{\omega\in\Omega^\kappa}F_\omega).$$
\begin{definition}
A cluster is a subset $\textbf{C}$ of $\Omega^\kappa$ such that $\bigcap_{\omega\in\textbf{C}} B(\omega)\not=\emptyset$. Let $t(\textbf{C})$ be the center of the smallest $B(\omega)$ in the cluster.
\end{definition}
The grid property and \eqref{e.76} imply that $\Omega^\kappa$ can be split in clusters such that the intervals $B(\omega)$ and $B(\omega')$ are pairwise disjoint if $\omega$, $\omega'$ belong to distinct clusters. Thus,
$$\sup_{v\in \Sigma_N}|\sum_{\omega\in \Omega^\kappa}\tilde{F}_\omega(x,y,v)|=\sup_{\textbf{C}}\sup_{v\in \Sigma_N}|\sum_{\omega\in \textbf{C}}\tilde{F}_\omega(x,y,v)|.$$
Note also that $\sum_{\omega\in\textbf{C}}F_\omega$ and $\sum_{\omega\in\textbf{C}'}F_\omega$ are orthogonal, for two distinct clusters $\textbf{C}$ and $\textbf{C}'$. Using these two facts, it suffices to prove that for each cluster $\textbf{C}$,
$$\|\sup_{v\in \Sigma_N}|\sum_{\omega\in \textbf{C}}\tilde{F}_\omega(x,y,v)|\|_2\lesssim\|F\|_2$$

Note that for each $v\in\Sigma_N$
$$\sum_{\omega\in\textbf{C}}\tilde{F}_\omega(x,y,v)=\sum_{\omega\in\textbf{C}\atop{v\in B(\omega)}}\tilde{F}_\omega(x,y,v)=$$
$$=\sum_{\omega\in\textbf{C}\atop{v\in B(\omega)}}\tilde{F}_\omega(x,y,t(\textbf{C}))+O(M^*F(x,y))\;\;\;\hbox{(by Lemma }\ref{L543})\;=$$$$
=P_{m}(\sum_{\omega\in\textbf{C}}\tilde{F}_\omega(x,y,t(\textbf{C})))+O(M^*F(x,y))$$
for some appropriate $m=m(v,\textbf{C})$, where $P_m$ denotes the Fourier projection on the ball with radius $2^m$.
The last equality above follows since  $\tilde{F}_\omega(x,y,v)$ is supported in frequency in $\omega$.
Recall that the maximal operator $\sup_m|P_mF|$ is bounded on $L^2$, by invoking standard maximal function arguments. Also, the operator $\sum_{\omega\in\textbf{C}}\tilde{F}_\omega(x,y,t(\textbf{C}))$ has a bounded multiplier and thus it is  bounded on $L^2$. This ends the proof.

\end{document}